\documentclass[11pt,a4paper]{amsart}

\usepackage{amssymb,amscd,amsmath,young, stmaryrd,tikz-cd,hyperref}
\usepackage{mathrsfs, mathdots}
\usepackage[mathcal]{eucal}
\usepackage{float,mathabx}
\usepackage{soul,caption}
\usepackage{cancel}
\usepackage[normalem]{ulem}
\usepackage{hyperref}
\usetikzlibrary{arrows}

\usepackage{longtable,array,multirow,makecell,graphicx}

\newtheorem*{rep@theorem}{\rep@title}
\newcommand{\newreptheorem}[2]{%
	\newenvironment{rep#1}[1]{%
		\def\rep@title{#2 \ref{##1}}%
		\begin{rep@theorem}}%
		{\end{rep@theorem}}}

\newreptheorem{theorem}{Theorem}

\newreptheorem{lemma}{Lemma}
\theoremstyle{plain}
\newtheorem{thm}{Theorem}[section]
\newtheorem{lem}[thm]{Lemma}
\newtheorem{pro}[thm]{Proposition}
\newtheorem{cor}[thm]{Corollary}
\newtheorem*{claim*}{Claim}

\theoremstyle{remark}

\newtheorem{exm}[thm]{Example}

\newtheorem*{acknowledgements}{Acknowledgements}

\numberwithin{equation}{section}

\numberwithin{table}{section}

\newcommand{\N}{\mathbb{N}}
\newcommand{\Z}{\mathbb{Z}}

\newcommand{\C}{\mathbb{C}}

\renewcommand{\epsilon}{\varepsilon}

\renewcommand{\phi}{\varphi}
\renewcommand{\theta}{\vartheta}

\DeclareMathOperator{\real}{Re}

\author{Seungjai Lee} \address{Fakult\"at f\"ur Mathematik,
	Universit\"at Bielefeld, D-33501 Bielefeld, Germany}
\email{seungjai.lee@math.uni-bielefeld.de}

\keywords{Rational generating functions, Hilbert quasipolynomials, roots}
\subjclass[2020]{05A15, 12D10, 13D40}

\begin{document}
	
	\title[Hilbert quasipolynomials and their roots]{Hilbert quasipolynomials and their roots}
	\begin{abstract}
	We investigate the roots of Hilbert quasipolynomials arising from certain rational generating functions.
	\end{abstract}
	\date{\today}

	\maketitle

	\thispagestyle{empty}
\section{Introduction}
\subsection{Motivation}
Let $P(t)$ be a generating function of the form
\begin{align}\label{k=1}
P(t)&=\frac{U(t)}{(1-t)^d},&U(1)\neq 0,
\end{align}where $d\in\N$ and $U(x)\in\C[x]$ is of degree $e\leq d-1$. Then there is a polynomial $H(x)\in\C[x]$ of degree $d-1$ such that 
\[P(t)=\sum_{n\geq0}H(n)t^n.\]
For historical reasons we call $P$ the \textit{Poincaré series} and $H$ the \textit{Hilbert polynomial}.

For $a\in\Z$, let 
\[h_a(x)=\begin{cases}(x+1)(x+2)\cdots(x+a),&a\geq1,\\
1,&a<1,\end{cases}\]
and
\[S_{a}=\{p\in\C[x]\setminus\{0\}|p=h_{a}v,v\in\C[x]\textrm{ has all its roots on }\real(x)=-(a+1)/2\}.\] 
In \cite{R-V/02}, Rodriquez-Villegas proved the following result:
\begin{thm}[\cite{R-V/02}]
	Let the notation be as above. Suppose that all the roots of $U$ are on the unit circle. Then $H\in S_{d-1-e}.$
\end{thm}
In other words, when we have a Poincaré series $P$ whose numerator $U$ has all its root on the unit circle, then the roots of its Hilbert polynomial $H$ are either ``trivial integer roots" $-1,\ldots,e+1-d$ or lying on a vertical line $\real(x)=-(a+1)/2$.

In this paper we want to generalise this for the following rational functions and \textit{quasipolynomials}.

\subsection{Hilbert quasipolynomials and rational generating functions}

Let $F:\Z\rightarrow\C$ be a function. We call $F$ a \textit{quasipolynomial} if there exist a positive integer $k$ and polynomials $F_{0}(n),\ldots,F_{k-1}(n)$ such that
\[F(n)=\begin{cases} F_{0}(n)&\textrm{if }n\equiv 0\mod k,\\
F_{1}(n)&\textrm{if }n\equiv 1\mod k,\\
\vdots&\\
F_{k-1}(n)&\textrm{if }n\equiv k-1\mod k.\\
\end{cases}\]
The minimal such $k$ is the \textit{period} of $F(n)$, and for this minimal $k$, the polynomials $F_{0}(n),\ldots,F_{k-1}(n)$ are the \textit{constituents} of $F(n)$.

There is a connection between quasipolynomials and generalised Poincaré series.
\begin{pro}[\cite{BS/18}, Proposition 4.5.1]
	Let $H:\Z\rightarrow\C$ be a function with associated generating function
	\[P(t):=\sum_{n\geq0}H(n)t^n.\]
	Then $H(n)$ is a quasipolynomial of degree $\leq d-1$	and period dividing $k$ if and only if 
	\[P(t)=\frac{U(t)}{(1-t^k)^{d}},\]
	where $U(t)$ is a polynomial of degree at most $kd-1$.
\end{pro}
So when $k=1$, we get the simplest form \eqref{k=1}, and a unique  Hilbert polynomial $H(n)$ for all $n\geq0$. 
\subsection{Main results}
Let $P(t)$ be a rational function as given above. For a given $U(t)=c_{0}+c_{1}t+\cdots+c_{e}t^{e}$ and $0\leq i \leq k-1$, let 
\[U_{i}(t)=c_{i}t^i+c_{k+i}t^{k+i}+\cdots+c_{q_{i}k+i}t^{q_{i}k+i}\] denote the polynomial that consists of all the terms in $U(t)$ of degree congruent to $i$ modulo $k$. If $U_i(t)\neq0$, we have a unique $q_i\in\N_0$ where $e_{i}=q_{i}k+i\leq e$ is the degree of $U_{i}(t)$ Otherwise the degree $e_i$ of $U_{i}(t)$  is $0$. 

Also, for $a\in\Z, k\in\N_+$, and $0\leq i\leq k$, let 
\[h_{a,k,i}(x)=\begin{cases}(x+k-i)(x+2k-i)\cdots(x+ak-i),&a\geq1,\\
1,&a<1,\end{cases}\]
and
\[S_{a,k,i}=\{p\in\C[x]\setminus\{0\}|p=h_{a,k,i}v,v\in\C[x]\textrm{ has all its roots on }\real(x)=-((a+1)k-2i)/2\}.\]
In this paper we prove the following results:
\begin{reptheorem}{thm:RV.quasi}
	Let \[P(t)=\frac{U(t)}{(1-t^k)^{d}}=\sum_{n\geq0}H(n)t^n,\] and let the notation be as above.
	For any $i$, suppose that $U_{i}(1)\neq0$ and all the non-zero roots are on the unit circle. Then $H_{i}\in S_{d-q_{i}-1,k,i}$.
\end{reptheorem}
\begin{reptheorem}{thm:quasi.global}
	Let $H_{\times}=\prod_{0\leq i\leq k-1}H_{i}$. If $U_{i}\neq0$ for all $i$, then $n\in\{-1,-2,\ldots,-(dk-e-1)\}$ satisfies $H_{\times}(n)=0$.
\end{reptheorem}
Theorem \ref{thm:RV.quasi} can be regarded as the generalised quasi-version of the result in \cite{R-V/02}. In particular, Theorem \ref{thm:RV.quasi}  proves \cite{R-V/02} when $k=1$. Theorem \ref{thm:quasi.global} gives some insights on how the Hilbert quasipolynomials of a given rational function behaves globally if you look all of them together.

\section{preliminaries}
In this section we prove some useful results to prove the main theorems.
\begin{lem}
	Let
	\[P(t)=\frac{c}{(1-t^k)^d}.\]
	Then we have $H_{0}(n)=\frac{c}{(d-1)!k^{d-1}}(n+k)(n+2k)\cdots(n+(d-1)k)$ and $H_{1}(n)=\cdots=H_{k-1}(n)=0$.
\end{lem}
\begin{proof}
	 $H_{1}(n)=\cdots=H_{k-1}(n)=0$ is trivial. We need to show that $H_{0}(n)=\frac{c}{(d-1)!k^{d-1}}(n+k)(n+2k)\cdots(n+(d-1)k)$.
	
	Note that one can write
	\begin{align*}
	P(t)&=\frac{c}{(1-t^k)^d}\\
	&=c\sum_{n=0}^{\infty}\binom{n+d-1}{n}t^{kn}.
	\end{align*}
	Since $\deg(H_0(n))\leq d-1$, $H_{0}(n)$ should satisfy
	\begin{align*}
	H_0(0)&=c\binom{0+d-1}{0}=c,\\
	H_0(k)&=c\binom{1+d-1}{1}=cd,\\
	H_0(2k)&=c\binom{2+d-1}{2}=\frac{c}{2}d(d+1),\\
	\vdots&\\
	H_0((d-1)k)&=c\binom{(d-1)+d-1}{d-1}=\frac{c}{(d-1)!}d(d+1)\cdots(2d-3)(2d-2).
	\end{align*}
	This is exactly the case when
	\[H_{0}(n)=\frac{c}{(d-1)!k^{d-1}}(n+k)(n+2k)\cdots(n+(d-1)k),\]
	
	since for $0\leq m\leq d-1$,
	\begin{align*}
	H_{0}(mk)&=\frac{c}{(d-1)!k^{d-1}}(mk+k)(mk+2k)\cdots(mk+(d-1)k)\\
	&=\frac{ck^{d-1}}{(d-1)!k^{d-1}}\prod_{i=1}^{d-1}(m+i)\\
	&=c\binom{m+d-1}{d-1}=c\binom{m+d-1}{m}
	\end{align*} 
	as required.
\end{proof}
\begin{exm}
	Let 
	\[P(t)=\frac{1}{(1-t^2)^2}=\sum_{n\geq0}H(n)t^n.\]
	Then 
	\[H(n)=\begin{cases}\frac{1}{2}(2+n)&n\equiv0\mod2,\\
	0&n\equiv1\mod2.
	\end{cases}\]
\end{exm}
\begin{cor}\label{cor:single.power}
	For $0\leq j \leq dk-1$, where $j\equiv r \mod k$ for $0\leq r \leq k-1$, let
	\[P(t)=\frac{c_jt^j}{(1-t^k)^d}.\]
	Then we have $H_{r}(n)=\frac{c_j}{(d-1)!k^{d-1}}(n+k-j)(n+2k-j)\cdots(n+(d-1)k-j)$ and all other $H_i(n)=0$.
\end{cor}
Another trivial but useful result is the following.
\begin{lem}\label{lem:sum}
	Let
	\begin{align*}
	P(t)&=\frac{U(t)}{(1-t^k)^d}=\sum_{n\geq0}H(n)t^n&P'(t)&=\frac{U'(t)}{(1-t^k)^d}=\sum_{n\geq0}H'(n)t^n.
	\end{align*}
	Then the Hilbert quasipolynomial of $P(t)+P'(t)$ is $H(n)+H'(n)$.
\end{lem}
\begin{proof}
	By definition $P(t)+P'(t)=\sum_{n\geq0}(H(n)+H'(n))t^n$.
\end{proof}
Hence in general, for a rational function
\[P(t)=\frac{U(t)}{(1-t^k)^{d}},\] one can write
\begin{align*}
P_i(t)&:=\frac{U_i(t)}{(1-t^{k})^{d}},&P(t)&=\sum_{i=0}^{k-1}P_i(t)=\sum_{i=0}^{k-1}\frac{U_i(t)}{(1-t^k)^d},
\end{align*}
where $U_i(t)$ denote the polynomial that only contains all the terms in $U(t)$ whose degree is $i\mod k$. Each constituent $H_i(n)$ of the Hilbert quasipolynomial $H(n)$ of $P(t)$ in fact only comes from each $P_{i}(t)$.
\section{Roots of Hilbert quasipolynomials}
 In this section we prove the two main theorems of this paper: Theorem \ref*{thm:RV.quasi} and Theorem \ref{thm:quasi.global}.
 
Start with the following lemma. For $a\in\Z, k\in\N_+$, and $0\leq i\leq k-1$, let 
 \[h_{a,k,i}(x)=\begin{cases}(x+k-i)(x+2k-i)\cdots(x+ak-i),&a\geq1,\\
 1,&a<1,\end{cases}\]
 and
 \[S_{a,k,i}=\{p\in\C[x]\setminus\{0\}|p=h_{a,k,i}v,v\in\C[x]\textrm{ has all its roots on }\real(x)=-((a+1)k-2i)/2\}.\]
 \begin{lem}\label{lem:realline}
 	Let $\alpha\in\C$ with $|\alpha|=1$, and $f\in S_{a,k,i}$ for some $a,k,i$. Then 
 	\begin{equation}
 	f(x-k)-\alpha f(x)\in S_{a-1,k,i}.
 	\end{equation}
 \end{lem}
 \begin{proof}
 	Note that $h_{a-1,k,i}$ divides bot $f(x)$ and $f(x-k)$. Let
 	\begin{align*}
 	f(x)&=h_{a-1,k,i}(x)r(x)\\
 	f(x-1)&=h_{a-1,k,i}(x)s(x)\\
 	g(x)&=s(x)-\alpha r(x).
 	\end{align*}
 	We need to prove that $g$ has all its roots on the line $\real(x)=-(ak-2i)/2$. By our definitions, for every root $\mu$ of $r(x)$ that lies on the left half plane $\real(x)<-(ak-2i)/2$ there is a corresponding root $\nu$ of $s(x)$ located symmetrically on the right half plane $\real(x)>-(ak-2i)/2$, and vice versa. Let $\beta\in\C$ a root of $g$. Then $g(\beta)=s(\beta)-\alpha r(\beta)=0$, giving \[|s(\beta)|=|r(\beta)|.\]
 	If $\real(\beta)<-(ak-2i)/2$ say, then
 	\[|\beta-\mu|<|\beta-\nu|,\]
 	for all pairs of corresponding roots $\mu,\nu$, giving a contradiction as
 	\[|r(\beta)|=\prod_{\mu}|\beta-\mu|<\prod_{\nu}|\beta-\nu|=|s(\beta)|.\]
 	An analogous arguments shows that there are also no roots with $\real(\beta)>-(ak-2i)/2$. Hence all the roots of $g$ lie on the line $\real(x)=-(ak-2i)/2$ as required.
 \end{proof}
With this lemma, we prove:
 \begin{thm}\label{thm:RV.quasi} 
 Let \[P(t)=\frac{U(t)}{(1-t^k)^{d}}=\sum_{n\geq0}H(n)t^n,\] and let the notation be as above.
 For any $i$, suppose that $U_{i}(1)\neq 0$ and all the non-zero roots are on the unit circle. Then $H_{i}\in S_{d-q_{i}-1,k,i}$.
 \end{thm}
 \begin{proof}
 	We proceed by induction on $q_i$. 
 	
 	For $q_i=0$, $U_i(t)=c_{i}t^{i}$ and Corollary \ref{cor:single.power} gives 
 	\[H_{i}(n)=\frac{c_{i}}{(d-1)!k^{d-1}}(n+k-i)(n+2k-i)\cdots(n+(d-1)k-i),\]
 	which belongs to $S_{d-1,k,i}$.
 	
 	Now, if $H_i\in S_{d-1-q_{i},k,i}$ is the Hilbert polynomial corresponding to $U_i(t)$, then one can check that the Hilbert polynomial corresponding to $(t^k-\alpha)U_{i}(t)=t^{i}(t^k-\alpha)U_{i}'(t)$ is $H_i^{*}(x)=H_i(x-k)-\alpha H_i(x)$. By hypothesis $|\alpha|=1$, and hence Lemma \ref{lem:realline} implies that $H_{i}^{*}\in S_{d-2-q_{i},k,i}$ as required.
 \end{proof}

As we discussed, this can be seen as a generalised extension of the result in \cite{R-V/02}. It shows how the roots of each \textit{local} constituent $H_{i}$ of Hilbert quasipolynomial $H$ are arranged nicely on one critical line $\real(x)=-((d-q_{i})k-2i)/2$, when all the non-zero roots of $U_{i}$ are on the unit circle. 

Furthermore, it turns out that we can say more about some of the integer roots of Hilbert quasipolynomials even when not all the non-zero roots of $U_{i}(t)$ are on the unit circle. 
 \begin{pro}\label{pro:integer.roots}
 	Let the notation be as above. Suppose $U_{i}(t)\neq0$ with degree $e_{i}=q_{i}k+i$. Then $-(k-i),-(2k-i),\ldots,-((d-1-q)k-i)$ are members of the roots of $H_{i}(n)=0$.
 	\end{pro}
 \begin{proof}
 	For each $i$, let $e_i=q_{i}k+i$. By applying Corollary \ref{cor:single.power} and Lemma \ref{lem:sum}, we have
\begin{align*}
 	H_{i}(n)&=\frac{1}{(d-1)!k^{d-1}}\sum_{a=0}^{q_{i}}c_{ak+i}(n+k-(ak+i))(n+2k-(ak+i))\cdots(n+(d-1)k-(ak+i))\\
 	&=\frac{C_i(n)}{(d-1)!k^{d-1}}\prod_{j=1}^{d-1-q_{i}}(n+jk-i),
\end{align*}
where $C_i(x)\in\C[x]$ since $\prod_{j=1}^{d-1-q_{i}}(n+jk-i)$ is a common factor for $(n+k-(ak+i))(n+2k-(ak+i))\cdots(n+(d-1)k-(ak+i))$.
Therefore $n=-(k-i),-(2k-i),\ldots,-((d-1-q)k-i)$ satisfies $H_i(n)=0$.
 \end{proof}
Note that Proposition \ref{cor:single.power} has no restriction on the roots of $U_{i}$. 
\begin{thm}\label{thm:quasi.global}
	Let $H_{\times}=\prod_{0\leq i\leq k-1}H_{i}$. If $U_{i}\neq0$ for all $i$, then $n\in\{-1,-2,\ldots,-(dk-e-1)\}$ satisfies $H_{\times}(n)=0$.
\end{thm}
\begin{proof}
Since $U_i(t)\neq0$ for any $i$, $H_{i}(n)\neq0$ and Proposition \ref{pro:integer.roots} implies that  $n=-(k-i),-(2k-i),\ldots,-((d-1-q_i)k-i)$ satisfies $H_i(n)=0$. As $e=\deg(U(t))$,  there has to be an integer $r$ such that $e=e_{r}=q_{r}k+r$, $0\leq q_0,\ldots,q_{r-1}\leq q_{r}$, and $0\leq q_{r+1},\ldots,q_{k-1}<q_{r}$. Therefore, looking at the roots of $H_{i}(n)$, we have at least $-1, -2, \ldots,-((d-1-(q_r-1))k-(r+1))=-(dk-q_rk-r-1)=-(dk-e-1)$ satisfying $H_{\times}(n)=0$ as required.
\end{proof}
Theorem \ref{thm:quasi.global} implies that no matter how your rational function 
\[P(t)=\frac{U(t)}{(1-t^k)^{d}}=\sum_{n\geq0}H(n)t^n\]
looks like, globally $H_{\times}(n)$ takes trivial integer roots $n\in\{-1,-2,\ldots,-(dk-e-1)\}$, which are only determined by the degree of the numerator and the denominator of $P(t)$. In particular this implies $h_{d-e-1}(n)\mid H(n)$ if $k=1$.
\begin{acknowledgements}
	I would like to thank Claudia Alfes-Neumann, Christopher Voll, and ZiF, Bielefeld for organizing a nice workshop "Hecke operators, Ehrhart theory, and automorphic forms" that motivated this paper. Also thank Kyeongmin Kim and Seok Hyeong Lee for computational helps.
\end{acknowledgements}

	\bibliographystyle{amsplain}
	\bibliography{Lee_quasipoly_arXiv}
\end{document}